\documentclass[12pt]{amsart}

\usepackage{amsmath,amssymb,enumerate}
\newtheorem{theorem}{Theorem}[section]
\newtheorem{claim}{Claim}[theorem]
\newtheorem{lemma}[theorem]{Lemma}

\newtheorem{conjecture}[theorem]{Conjecture}
\theoremstyle{definition}

\newcommand{\eps}{\varepsilon}

\newcommand{\bZ}{\mathbb Z}
\newcommand{\bN}{\mathbb N}

\newcommand{\cN}{\mathcal{N}}

\newcommand{\bfj}{\mathbf{j}}
\newcommand{\lb}{\left(}
\newcommand{\rb}{\right)}
\newcommand{\lt}{\left}
\newcommand{\rt}{\right}

\newcommand{\cn}{\chi}

\newcommand{\cNq}{\cN_{1/4}}
\newcommand{\sqbinom}[2]{{#1 \brack #2}}
\newcommand{\cNh}{\cN_{1/2}}
\DeclareMathOperator{\wt}{wt}

\DeclareMathOperator{\cl}{cl}

\DeclareMathOperator{\codim}{codim}
\DeclareMathOperator{\PG}{PG}

\DeclareMathOperator{\GF}{GF}

\DeclareMathOperator{\rank}{rank}
\DeclareMathOperator{\supp}{supp}
\DeclareMathOperator{\ev}{\mathbf{E}}
\DeclareMathOperator{\spn}{span}

\newcommand{\del}{\!\setminus\!}

\begin{document}

\title[Triangle-free binary matroids]{The critical number of dense triangle-free binary matroids}
\author[Geelen]{Jim Geelen}
\address{Department of Combinatorics and Optimization,
University of Waterloo, Waterloo, Canada}
\thanks{This research was partially supported by a grant from the
Office of Naval Research [N00014-10-1-0851].}

\author[Nelson]{Peter Nelson}

\subjclass{05B35}
\keywords{matroids, regularity}
\date{\today}
\begin{abstract}
We show that, for each real number $\eps > 0$ there is an integer $c$ such that,
if $M$ is a simple triangle-free binary matroid
with $|M| \ge \lt(\tfrac{1}{4} + \eps\rt) 2^{r(M)}$, then $M$ has critical number at most $c$.
We also give a construction showing that no such result holds when replacing $\tfrac{1}{4} + \eps$ with $\tfrac{1}{4}-\eps$ in this statement. 
This shows that the ``critical threshold" for the triangle is $\tfrac 1 4$.
We extend the notion of critical threshold to every simple binary matroid $N$ and
conjecture that, if $N$ has critical number $c\ge 3$, then
$N$ has critical threshold $1-i\cdot 2^{-c}$ for some $i\in \{2,3,4\}$.
We give some support for the conjecture by establishing lower bounds. 
\end{abstract}

\maketitle

\section{Introduction}

If $M$ is a simple binary matroid, viewed as a restriction of a rank-$r$ projective geometry $G \cong \PG(r-1,2)$, then the \emph{critical number} of $M$, denoted $\cn(M)$, is the minimum nonnegative integer $c$ such that $G$ has a rank-$(r-c)$ flat disjoint from $E(M)$.
A matroid with no $U_{2,3}$-restriction is \emph{triangle-free}. Our first two main theorems are the following:

\begin{theorem}\label{mainplus}
	For each $\eps > 0$ there exists $c \in \bZ$ such that every simple triangle-free binary matroid $M$ with $|M| \ge \lt(\tfrac{1}{4} + \eps\rt)2^{r(M)}$ satisfies $\cn(M) \le c$. 
\end{theorem}

\begin{theorem}\label{mainminus}
	For each $\eps > 0$ and each integer $c \ge 1$, there is a simple triangle-free binary matroid $M$ such that $|M| \ge \lt(\tfrac{1}{4}-\eps\rt)2^{r(M)}$ and $M$ has critical number $c$. 
\end{theorem}

That is, simple triangle-free binary matroids with density slightly more than $\tfrac{1}{4}$ have bounded critical number, and those with density slightly less than $\tfrac{1}{4}$ can have arbitrarily large critical number. Theorem~\ref{mainminus} refutes an earlier conjecture of the authors [\ref{gn14}]. As in [\ref{gn14}], the proof of Theorem~\ref{mainplus} depends on a regularity lemma due to Green [\ref{green}]; this material is discussed in Section~\ref{fourier}.

 The critical number was originally defined by Crapo and Rota [\ref{cr}] under the name of \emph{critical exponent}; our terminology follows Welsh [\ref{welsh}]. One can also define $\chi(M)$ as the minimum $c$ so that $E(M)$ is contained in a matroid whose ground set is the union of $c$ affine geometries. In particular, if $M$ is the cycle matroid of a graph $G$, then $\chi(M)$ is the minimum number of cuts required to cover $E(G)$, so $\chi(M) = 1$ precisely when $G$ is bipartite, and $\chi(M) = \lceil \log_2 (\chi(G)) \rceil$ in general. Thus, we can view critical number as a geometric analog of chromatic number; results in graph theory motivate much of the material in this paper.
 
In analogy to our two main theorems, Hajnal (see [\ref{erdossimonovits}]) gave examples of triangle-free graphs $G$ with minimum degree $\delta(G) \ge \lt(\tfrac{1}{3}-\eps\rt)|V(G)|$ and arbitrarily large chromatic number, and Thomassen [\ref{t02}] showed for each $\eps > 0$ that every triangle-free graph $G$ with $\delta(G) \ge \lb \tfrac{1}{3}+ \eps\rb |V(G)|$ has chromatic number bounded above by a function of $\eps$. 

In fact, something much stronger holds; in [\ref{bt}], Brandt and Thomass\'{e} showed that if $G$ is a  triangle-free graph $G$ with minimum degree $\delta(G) > \tfrac{1}{3}|V(G)|$, then $\chi(G) \in \{2,3,4\}$. The bound $\chi(G)\le 4$ is best possible;
H\"aggkvist [\ref{29}] found an example of a $10$-regular triangle-free graph on $29$ vertices with chromatic number $4$. 
We conjecture a similar strengthening of Theorem~\ref{mainplus}.
\begin{conjecture}\label{goodbound}
	If $M$ is a simple triangle-free binary matroid with $|M| > \tfrac{1}{4}2^{r(M)}$, then $\cn(M) \in \{1,2\}$. 
\end{conjecture}

\subsection*{Chromatic threshold}
Erd\H os and Simonovits~[\ref{erdossimonovits}] proposed the problem, for a given simple graph $H$ and $\alpha>0$,
of determining the maximum of $\chi(G)$ among all $H$-free graphs $G$ with
minimum degree at least $\alpha |V(G)|$. Extending on this idea, 
\L uczak and Thomass\' e~[\ref{tl}]
define the \emph{chromatic threshold} for $H$ to be the infimum of all $\alpha > 0$ such that there exists $c = c(H,\alpha)$ for which every graph $G$ with no $H$-subgraph and with minimum degree at least $\alpha|V(G)|$ has chromatic number at most $c$.

The aforementioned results for the triangle $C_3$ give that its chromatic threshold is $\tfrac{1}{3}$.
 The Erd\H os-Stone Theorem [\ref{es}] implies that the chromatic threshold for any bipartite graph $H$ is $0$, since large dense $H$-free graphs do not exist.  Quite remarkably,  
 the chromatic thresholds of all graphs have been explicitly determined 
 by Allen et al. in~[\ref{abgkm}];
 here we will state a simplified version of their result that limits the threshold to one of three particular values
 depending only on $\chi(H)$.
 \begin{theorem}\label{graphthresholds}
 	If $H$ is a graph of chromatic number $c \ge 3$, then  $H$ has chromatic threshold in $\left\{\frac{c-3}{c-2},\frac{2c-5}{2c-3},\frac{c-2}{c-1}\right\}$.
 \end{theorem}

\subsection*{Critical threshold}
For a simple binary matroid $N$, we define the \emph{critical threshold} of $N$ to be the infimum of all $\alpha > 0$ such that there exists $c = c(N,\alpha)$ for which every simple binary matroid $M$ with no $N$-restriction and with $|M| \ge \alpha 2^{r(M)}$ satisfies $\cn(M) \le c$. For each integer $k\ge 3$, let $C_k$ denote the $k$-element circuit $U_{k-1,k}$. Theorems~\ref{mainplus} and~\ref{mainminus} imply that 
the critical threshold for $C_3$ is $\tfrac{1}{4}$. 
In contrast, the main result of~[\ref{gn14}] shows that, if $k \ge 5$ is odd, then $C_k$ has critical threshold $0$.

A result of Bonin and Qin [\ref{bq}], itself a special case of the geometric density Hales-Jewett theorem [\ref{fk2}], implies that 
each simple binary matroid with critical number $1$ has critical threshold $0$.
More generally, the geometric Erd\H os-Stone theorem~[\ref{gn12}] gives the following
upper bound on the critical threshold of any simple binary matroid.

\begin{theorem}\label{threshupperbound}
	The critical threshold for a simple binary matroid $N$ is at most $1 - 2^{1-\cn(N)}$. 
\end{theorem}

We show, in fact, that this holds with equality fairly often.

\begin{theorem}\label{richcase}
	If $N$ is a simple binary matroid of critical number $c \ge 1$ so that $\cn(N \del I) = c$ for every rank-$(n-c+1)$ independent set $I$ of $N$, then the critical threshold for $N$ is $1 - 2^{1-c}$. 
\end{theorem}

In Conjectures~\ref{crit2} and~\ref{critc}, we predict the precise value of the critical threshold for 
any simple binary matroid. The following is a simplification of those conjectures
in the vein of Theorem~\ref{graphthresholds}. 

\begin{conjecture}\label{threshvals}
	If $N$ is a simple nonempty binary matroid, then the critical threshold for $N$ is equal to $1 - i \cdot 2^{-\cn(N)}$ for some $i \in \{2,3,4\}$. 
\end{conjecture}

Specialised to projective geometries, our conjectures give:

\begin{conjecture}\label{pg}
	For each $t \ge 2$, the critical threshold for $\PG(t-1,2)$ is $1 - 3 \cdot 2^{-t}$. 
\end{conjecture}

Finally, we pose the following strengthening of Conjectures~\ref{goodbound} and~\ref{pg}; 
the analogous result was proved for graphs by Goddard and Lyle in [\ref{gl}].

\begin{conjecture}
	If $t \ge 2$ and $N$ is a simple binary matroid with no $\PG(t-1,2)$-restriction such that $|N| > (1-3 \cdot 2^{-t})2^{r(N)}$, then $\cn(N) \in \{t-1,t\}$.
\end{conjecture}

\section{Regularity}\label{fourier}

Green used Fourier-analytic techniques to prove his regularity
lemma for abelian groups and to derive applications
in additive combinatorics; these techniques are discussed in greater 
detail in the book of Tao and Vu~[\ref{tv06}, Chapter 4]. 
Fortunately, although this theory has many technicalities,
the group $\GF(2)^n$ is among its simplest applications. 

Let $V = \GF(2)^n$ and let $X\subseteq V$.
Note that, if $H$ is a $1$-codimensional subspace of $V$,
then $|H| = |V \del H|$.
We say that $X$ is {\em $\eps$-uniform} if
for each $1$-codimensional subspace $H$ of $V$ we have
$$ | \,|H\cap X| - |X\del H|\, | \le \eps |V|.$$
In Lemma~\ref{sumlemma} we will see that,
for small $\eps$, the $\eps$-uniform sets are `pseudorandom'.

Let $H$ be a subspace of $V$.
For each $v \in V$, let $H_v(X) = \{h \in H: h + v \in X\}$.
For $\eps > 0$, we say $H$ is \emph{$\eps$-regular}
with respect to $V$ and $X$ if $H_v(X)$ is $\eps$-uniform
in $H$ for all but $\eps |V|$ values of $v \in V$.

Regularity captures the way that $X$ is distributed among
the cosets of $H$ in $V$.  For $v\in V$,
we let $X+v = \{x+v\, : \, x\in X\}$;
thus $X+v$ is a translation of $X$. Note that
$X+v$ is $\eps$-uniform if and only if $X$ is.
Also note that $H_v(X) + v = X\cap H'$ where $H'=H+v$
is the coset of $H$ in $V$ that contains $v$.
Therefore, if $u,v\in H'$, then
$H_u(X)$ and $H_v(X)$ are translates of one another.
So $H$ is $\eps$-regular if, for all but an 
$\eps$-fraction of cosets $H'$ of $H$,
the set $(H'\cap X)+v$ is $\eps$-uniform  in $H$ for some $v\in H'$.

The following result of Green [\ref{green}] guarantees a regular
subspace of bounded codimension. Here $T(\alpha)$ denotes an
exponential tower of $2$'s of height $\lceil \alpha \rceil$.

\begin{lemma}[Green's regularity lemma] \label{regularity}
Let $X$ be a set of points in a vector space $V$ over $\GF(2)$ and let $0 < \eps  < \tfrac{1}{2}$.
Then there is  a
subspace $H$ of $V$, having codimension at most $T(\eps^{-3})$, that is $\eps$-regular with respect to $X$
and $V$. 
\end{lemma}

If $A_1,A_2,A_3$ were random subsets of $\GF(2)^n$ with 
$|A_i| = \alpha_i 2^n$, we would expect approximately $\alpha_1\alpha_2\alpha_3 2^{2n}$ 
solutions to the linear equation $a_1 + a_2 + a_3 = 0$ with $a_i \in A_i$. The next lemma,
found in [\ref{green}] and also a corollary of [\ref{tv06}, Lemma 4.13], bounds the error
in such an estimate when at least two of these sets are uniform. 
	
\begin{lemma}\label{sumlemma}
Let  $V$ be an $n$-dimensional vector space over $\GF(2)$, and let
$A_1,A_2,A_3 \subseteq V$ with $|A_i|=\alpha_i |V|$.
If $0 < \eps < \tfrac{1}{2}$ and $A_1$ and $A_2$ are $\eps$-uniform, then
\[ \lt|\{(a_1,a_2,a_3) \in A_1 \times A_2 \times A_3 : a_1 + a_2 + a_3 = 0\}\rt| \ge (\alpha_1\alpha_2\alpha_3-\eps) 2^{2n}. \]
\end{lemma}

\section{Triangle-free binary matroids}

We mostly use standard notation from matroid theory [\ref{oxley}]. 
It will also be
convenient to think of a simple rank-$n$ binary matroid as a subset of the vector space $V = \GF(2)^n$. 
For $X \subseteq V - \{0\}$,  we write $M(X)$ for the simple binary matroid on $X$ represented by a binary matrix with column set $X$. 

We require an easy lemma about triples of vectors with sum zero.

\begin{lemma}\label{trifreevectors}
	If $X$ is a set of elements in an $n$-dimensional vector space $V$ over $\GF(2)$ with $|X| > 2^{n-1}$, then for all $v \in V$ there exist $x_1,x_2 \in X$ such that $x_1 + x_2 + v = 0$. 
\end{lemma}
\begin{proof}
	If $v = 0$, the result is trivial. If $v \ne 0$; the elements of $V$ partition into $2^{n-1}$ pairs $(x,y)$ with $x+y+v = 0$. Since $|X| > 2^{n-1}$, some such pair contains two elements of $X$, giving the result. 
\end{proof}

We now prove Theorem~\ref{mainplus} by means of the following stronger result, which shows that the theorem holds not just for triangle-free matroids but for all matroids in which each element is in $o(2^r)$ triangles.

\begin{theorem}
	For each $\eps >0$ there exist $c \in \bZ$ and $\beta >0$ such that, if $M$ is a simple binary matroid with $|M| \ge (\tfrac{1}{4} + \eps)2^{r(M)}$, then either $\cn(M) \le c$, or there is some $e \in E(M)$ contained in at least $\beta 2^{r(M)}$ triangles of $M$. 
\end{theorem}
\begin{proof}
	We may assume that $\eps < \tfrac{3}{4}$. Let $\delta = \tfrac{1}{16}\eps^3$, noting that $\delta < \tfrac{1}{2}$ and $(1 + 2\delta)^2 < 1 + 2\eps$, and set $c \ge T(\delta^{-3})$. Let $\beta = 2^{-2c}\delta.$
	
	Let $M$ be a simple rank-$r$ binary matroid with $|M| \ge \left(\tfrac{1}{4} + \eps\right)2^{r(M)}$. Let $V = \GF(2)^r$ and $X \subseteq V$ be such that $M = M(X)$.  Suppose that each $e \in E(M)$ lies in at most $\beta 2^{r(M)}$ triangles of $M$.
	
Since $\delta < \tfrac{1}{2}$, by Lemma~\ref{regularity} there is a subspace $H$ of $V$ that is $\delta$-regular with respect to $X$ and $V$ and has codimension $k \le c$ in $V$. If $X \cap H = \varnothing$ then $\chi(M) \le k \le c$, giving the theorem, so we may assume that there is some $v_0 \in X \cap H$.  Let $W$ be the subspace of $V$ that is `orthogonal' to $H$; thus $|W| = 2^k$ and $\{H + w: w \in W\}$ is the collection of cosets of $H$ in $V$. We first claim that $X$ is not too dense in any coset:
	
\begin{claim}
$|X \cap (H+w)| \le \left(\tfrac{1}{2}+\delta\right)2^{r-k}$ for each $w \in W$. 
\end{claim}
\begin{proof}[Proof of claim:]
The elements of $H+w$ partition into $2^{r-k-1}$ pairs adding to $v_0$; since the element of $M$ corresponding to $v_0$ is in at most $\beta 2^r$ triangles of $M$, at most $\beta 2^r$ of these pairs contain two elements of $X$. (This also holds for $w = 0$ since $0 \notin X$.) Therefore \[|(H+w) \cap X| \le 2^{r-k-1} + \beta 2^r \le \left(\tfrac{1}{2} + 2^k\beta\right)2^{r-k} \le \left(\tfrac{1}{2} + \delta\right)2^{r-k},\]
		as required. 
	\end{proof}
	Let $Z = \{w \in W: |X \cap (H + w)| \ge \tfrac{\eps}{2}2^{r-k}\}$.

	\begin{claim}
		$|Z| > \left(\tfrac{1}{2} + \delta\right)2^k$. 
	\end{claim}
	\begin{proof}[Proof of claim:] Using the first claim and $|W \setminus Z| \le 2^k$, we have
		\begin{align*}
			\left(\tfrac{1}{4} + \eps\right)2^r &\le |X| \\
									 &= \sum_{w \in W}|X \cap (H+w)|\\
									 &\le \sum_{w \in Z}\left(\tfrac{1}{2} + \delta\right)2^{r-k} + \sum_{w \in W \setminus Z}\tfrac{\eps}{2}2^{r-k}\\
									 &\le 2^{r-k}\left(\left(\tfrac{1}{2} + \delta\right)|Z| + \tfrac{\eps}{2}2^k\right).
		\end{align*}
		Thus $|Z| \ge \tfrac{1+2\eps}{2(1+2\delta)}2^k > \left(\tfrac{1}{2} + \delta\right)2^k$, where we use $(1+2\delta)^2 < 1 +2 \eps$.
	\end{proof}
	By regularity there are at most $\delta 2^k$ values of $w \in W$ such that $H_w(X)$ is not $\delta$-uniform, so there is a set $Z' \subseteq Z$ such that $|Z'| > 2^{k-1}$ and $H_w(X)$ is $\delta$-uniform for each $w \in Z'$. By Lemma~\ref{trifreevectors}, there are elements $w_1,w_2,w_3 \in Z'$ such that $w_1 + w_2 + w_3 = 0$. The sets $H_{w_1}(X), H_{w_2}(X),H_{w_3}(X)$ are $\delta$-uniform subsets of $H$ with at least $\tfrac{1}{2}\eps 2^{r-k}$ elements; by Lemma~\ref{sumlemma} the number of solutions to $x_1 + x_2 + x_3 = 0$ , so that $x_i \in H_{w_i}(X)$ for each $i \in \{1,2,3\}$, is at least $\left(\left(\tfrac{1}{2}\eps\right)^3-\delta\right)2^{2(r-k)} = \delta 2^{-2k} 2^{2r} \ge \beta 2^{2r}$. For any such solution, the vectors $x_1+w_1,x_2+w_2,x_3+w_3$ are elements of $X$ summing to zero,   so $M$ has at least $\beta 2^{2r}$ triangles. It follows, since $|M| < 2^r$, that some $e \in E(M)$ is in more than $\beta 2^r$ triangles, a contradiction. 
\end{proof}

\subsection*{The lower bound}
Theorem~\ref{mainplus} establishes an upper bound of $\tfrac 1 4$ on the critical
threshold of $C_3$. We have yet to prove Theorem~\ref{mainminus} which
gives the corresponding lower bound. We will in fact prove a stronger result,
Theorem~\ref{mainlb}. However, in the generalisation, we lose the simplicity
of the construction that works for $C_3$, so we give that construction 
here. The construction is very close to that of a `niveau set' 
(see [\ref{green04}], Theorem 9.4). 

	Let $c,n\ge 0$ be integers. Let $X_n$ denote the set of vectors in $\GF(2)^{n+1}$ with first entry zero and Hamming weight greater than $n-c$. Let $Y_n$ denote the set of vectors in $\GF(2)^{n+1}$ with first entry $1$ and Hamming weight at most $\tfrac{1}{2}(n-c)$. Let $M_{c,n}$ denote the matroid $M(X_n \cup Y_n)$. The following lemma implies Theorem~\ref{mainminus}.
	
\begin{lemma}
Let $c\ge 0$ be an integer and $\eps > 0$.
Then, for each sufficiently large integer $n$, the matroid
$M = M_{c,n}$ is triangle-free, has critical number $c+1$, and satisfies $|M| \ge (\tfrac{1}{4}- \eps)2^{r(M)}$.
\end{lemma}
\begin{proof}
	Suppose that $n > 3c$. Clearly $(Y_n + Y_n) \cap X_n$ and $(X_n + X_n) \cap X_n$ are empty; it follows that $M$ is triangle-free. By Stirling's approximation, $\max_{0 \le i \le n}\binom{n}{i} \le \binom{n}{\lfloor n/2 \rfloor} = O(\tfrac{2^{n}}{\sqrt{n}}) = o(2^n)$, so
	\[|Y_n| = \sum_{i=0}^{\left\lceil n/2 \right\rceil} \binom{n}{i} - \sum_{i= \left\lfloor (n-c)/2 \right\rfloor}^{\lceil n/2\rceil} \binom{n}{i} \ge \tfrac{1}{2} 2^{n} - \tfrac{c}{2} o(2^n);\] 
	since $r(M) = n+1$ and $|M| \ge |Y_n|$, this implies the required lower bound on $|M|$ for sufficiently large $n$. Let $b_1, \dotsc, b_{n+1}$ be the standard basis for $\GF(2)^{n+1}$ and let $\bfj = \sum b_i$. If $W = \spn(\{b_2, \dotsc, b_{n+1-c}\})$, then $\codim(W) = c+1$ and $W \cap E(M) = \varnothing$, so $\chi(M) \le c+1$. 
	
	Finally, we show that $\chi(M) > c$. Let $U$ be a subspace of $\GF(2)^{n+1}$ with $\codim(U) \le c$ and let $A$ be a matrix with at most $c$ rows having null space $U$. If there is some $y \in U$ with first entry $1$, then there exists $x \in \GF(2)^{n+1}$ with first entry zero and Hamming weight at most $\rank(U) \le c$ such that $Ax = A(y + b_1)$, giving $A(x+b_1) = Ay = 0$. Now $x + b_1$ has first entry $1$ and Hamming weight at most $c + 1 < \tfrac{1}{2}(n-c)$, so $x+b_1 \in U \cap Y_n$ and therefore $U \cap E(M) \ne \varnothing$. Suppose, therefore, that every $y \in U$ has first entry zero. Now there is a vector $z \in \GF(2)^{n+1}$ of Hamming weight at most $c$ such that $Az = A\bfj$; we have $z+\bfj \in U$ (and therefore $z + \bfj$ has first entry zero) and $z + \bfj$ has Hamming weight at least $n+1-c$, so $z + \bfj \in X_n \cap U$, again giving $U \cap E(M) \ne \varnothing$. This completes the proof.
\end{proof}

\section{Large girth and critical number}

Jaeger [\ref{j81}] gave a constructive characterisation of matroids with large critical number. 
Erd\H os~[\ref{gcn}] used a probabilistic argument to prove the existence of graphs with large girth and
chromatic number, which, since $\chi(M(G)) = \lceil \log_2(\chi(G)) \rceil$ for each graph $G$, gives binary matroids with large girth and critical number. 
We will use the probabilistic method to construct
such matroids with the additional property that they have a representation comprising only
vectors of large support.

For $x \in \GF(2)^S$, let $\supp(x)$ denote the support of $x$: that is, the set of all $s \in S$ such that $x_s \ne 0$. Let $\wt(x) = |\supp(x)|$ denote the Hamming weight of $x$. 
We require the following technical lemma, concerning vectors of small Hamming weight.

\begin{lemma}\label{nullspace}
	Let $c,s,n \in \bZ$ with $n \ge 2^{c+1}s$ and $s > c$, and let $W$ be a $c \times n$ binary matrix. 
For each $v \in \GF(2)^n$, the number of vectors $x \in \GF(2)^n$ satisfying $Wx = Wv$ and $\wt(x) \le s$ is at least $\lb\frac{n}{2^{c+1}s}\rb^{s-c-1}$. 
\end{lemma}
\begin{proof}
	Let $[n] = \{1, \dotsc, n\}$ index the column set of $W$. Since $Wv$ is in the column space of $W$, there is a vector $v_0\in \GF(2)^n$ with  $\wt(v_0) \le \rank(W) \le c$
	such that  $Wv_0 = Wv$; let $I = \supp(v_0) \subseteq [n]$. The matrix $W$ has at most $2^c$ distinct columns, so there is a set $J \subseteq [n]-I$ and a vector $w_0 \in \GF(2)^c$ such that $W_j = w_0$ for each $j \in J$ and
	\[ |J| \ge 2^{-c}([n]-|I|) \ge 2^{-c}(n-c) \ge 2^{-c-1}n \ge s. \] 
	If $s - |I|$ is even, then each vector $x$ such that $\wt(x) = s$ and $I \subseteq \supp(x) \subseteq I \cup J$  satisfies $Wx = Wv_0 + (s-|I|)w_0 = Wv$. If $s - |I|$ is odd, then each vector $x$ such that $\wt(x) = s-1$ and $I \subseteq \supp(x) \subseteq I \cup J$ satisfies $Wx = Wv_0 + (s-|I|-1)w_0 = Wv$. The number of vectors $x$ with $\wt(x) \le s$ and $Wx = Wv$ is therefore at least 
	\[ \min\lb \binom{|J|}{s-|I|}, \binom{|J|}{s-1-|I|}\rb \ge \lb\frac{|J|}{s}\rb^{s-|I|-1} \ge \lb\frac{n}{2^{c+1}s}\rb^{s-c-1},\]
	as required.
\end{proof}

The following lemma gives a subset of $\GF(2)^n$ of high girth and critical number, such that every vector has very large Hamming weight. 

\begin{lemma}\label{bespokegc}
	For all integers $c,g \ge 2$ and all sufficiently large $n \in \bZ$, there is a set $Z \subseteq \GF(2)^n$ such that  $M(Z)$ has girth at least $g$ and critical number at least $c$, and $\wt(z) \ge n - 2cg$ for each $z \in Z$. \end{lemma}
\begin{proof}
	Let $s = 2cg$ and let $\mu = 2^{c(c-s)}s^c$. Let $n$ be a sufficiently large integer such that $n \ge s$ and $(2s^s)^{-1/g} n^{2c} \ge c \mu^{-1}n^{c+1}+1$. We show that the result holds for $n$. 
	
	 Let $S$ be the set of vectors in $\GF(2)^n$ of Hamming weight at least $n-s$ and let $m = \lt\lfloor \lb \tfrac{1}{2}|S| \rb^{1/g} \rt\rfloor$. Using $|S| \ge \left(\frac{n}{s}\right)^s$ and our choice of $n$, we have \[m \ge (\tfrac{1}{2s^s})^{1/g} n^{s/g}-1 = (2s^s)^{-1/g} n^{2c} - 1 \ge c\mu^{-1}n^{c+1}.\]

	
	For each $m$-tuple $X = (x_1, \dotsc, x_m) \in S^m$ and each integer $k \ge 3$, let $\gamma_k(X)$ be the number of sub-$k$-tuples of $X$ that sum to zero.  Let $\gamma(X) = \sum_{k=3}^{g-1}\gamma_k(X)$; that is, $\gamma(X)$ is the number of `ordered circuits' of length less than $g$ contained in $X$. Similarly, let $\zeta(X)$ denote the number of $(c-1)$-codimensional subspaces of $\GF(2)^n$ that contain no element of $X$. Note that if $\gamma(X) = \zeta(X) = 0$, then the set $Z$ of elements in $X$ has critical number at least $c$ and contains no small circuits, so satisfies the lemma. We show with a probabilistic argument that the required $m$-tuple $X$ exists.  
	
	Let $X = (x_1, \dotsc, x_m)$ be an $m$-tuple drawn uniformly at random from $S^m$. Since the last element in any $k$-tuple in $S^k$ summing to zero is determined by the others, the probability that a $k$-tuple chosen uniformly at random from $S^{k}$ sums to zero is at most $|S|^{-1}$, so we have $\ev(\gamma_k(X)) \le m^k|S|^{-1}$ for each $k$. By linearity, we have 
	\[\ev(\gamma(X)) \le |S|^{-1}\sum_{k=3}^{g-1} m^k < m^g|S|^{-1} \le \frac{1}{2}.\]
	We now consider $\zeta(X)$. Let $F$ be an $(c-1)$-codimensional subspace of $\GF(2)^n$ and let $W$ be a $(c-1) \times n$ binary matrix with null space $F$. If $v$ is a vector chosen uniformly at random from $S$, then $v = v' + \bfj$, where $\bfj$ is the all-ones vector and $v'$ is chosen uniformly at random from $S'$, the set of vectors in $\GF(2)^n$ of Hamming weight at most $s$. We have $v' + \bfj \in F$ if and only if $Wv' = W\bfj$. By Lemma~\ref{nullspace}, the probability that $Wv' = W\bfj$ is at least 
	\[\frac{1}{|S'|}\lb \frac{n}{2^c s} \rb ^{s-c} \ge \left(\frac{s}{n}\right)^s\frac{n^{s-c}}{2^{c(s-c)}s^{s-c}} = \mu n^{-c}.\]
	Therefore the probability that $x_i \notin F$ for all $i \in \{1, \dotsc, m\}$ is at most $(1 - \mu n^{-c})^m$; since there are at most $2^{(c-1)n}$ subspaces $F$ of codimension $c-1$, it follows that 
	\begin{align*}
		\ev(\zeta(X)) \le 2^{(c-1)n}(1-\mu n^{-c})^m \le 2^{(c-1)n}\lb 2^{-\mu n^{-c}}\rb ^m,			
	\end{align*}
	Now, using $m \ge c\mu^{-1} n^{c+1}$, we have $(c-1)n - m\mu n^{-c} \le -n \le -1$.
	Therefore $\ev(\zeta(X)) \le \tfrac{1}{2}$. This gives $\ev(\gamma(X) + \zeta(X)) < 1$, so the required tuple $X_0$ with $\gamma(X_0) = \zeta(X_0) = 0 $ exists. 
	\end{proof}

\section{Critical thresholds}\label{lowerboundsection}

	We now formulate a conjecture predicting the critical threshold for every simple binary matroid, and prove that this prediction is a correct lower bound. To state the conjecture, we use a piece of new terminology. If $k \ge 0$
	is an integer and $M$ is a simple rank-$n$ binary matroid, viewed as a restriction of $G \cong \PG(n-1,2)$, then a \emph{$k$-codimensional subspace} of $M$ is a set of the form $F \cap E(M)$, where $F$ is a rank-$(n-k)$ flat of $G$. Such a set is a flat of $M$ and has rank at most $n-k$, but can also have smaller rank; for example, $\varnothing$ is a $1$-codimensional subspace of any simple binary matroid of critical number $1$.
	
	 Let $\cN$ denote the class of simple binary matroids of critical number $2$; we partition $\cN$ into three subclasses as follows:
 
	\begin{itemize} 
					\item Let $\cN_0$ denote the class of all $N \in \cN$ having a $1$-codimensional subspace $S$ such that $S$ is independent in $N$, and each odd circuit of $N$ contains at least four elements of $E(N)-S$. 
					\item Let $\cNq$ denote the class of all $N \in \cN - \cN_0$ so that some $1$-codimensional subspace of $N$ is independent in $N$.
					\item Let $\cNh = \cN - (\cN_0 \cup \cNq)$. 
	\end{itemize}

	We know from Corollary~\ref{threshupperbound} that binary matroids of critical number $1$ have critical threshold $0$. Our first conjecture predicts the threshold for the binary matroids of critical number $2$.	

\begin{conjecture}\label{crit2}
	For $\delta \in \{0,\tfrac{1}{4},\tfrac{1}{2}\}$, each matroid in $\cN_{\delta}$ has critical threshold $\delta$. 
\end{conjecture}

Note that every simple binary matroid $N$ of critical number $c \ge 2$ has a $(c-2)$-codimensional subspace $F$ such that $\cn(N|F) = 2$. Thus, the minimum in the following conjecture is well-defined, and the conjecture, which clearly implies Conjecture~\ref{threshvals}, predicts the critical threshold for every simple binary matroid of critical number at least $2$. 

\begin{conjecture}\label{critc}
	If $N$ is a simple binary matroid of critical number $c \ge 2$, then the critical threshold for $N$ is $1 - (1-\delta)2^{2-c}$, where  $\delta \in \lt\{0,\tfrac{1}{4},\tfrac{1}{2}\rt\}$  is minimal such that $N|S \in \cN_{\delta}$ for some $(c-2)$-codimensional subspace $S$ of $N$.
\end{conjecture}

Theorem~\ref{mainlb} will show that the value given by the above conjecture is a correct lower bound for the critical threshold. The next lemma deals with the case when $N$ has critical number $2$.
\begin{lemma}\label{crit2thm}
	Let $\delta \in \lt\{0,\tfrac{1}{4},\tfrac{1}{2}\rt\}$. For all integers   $c,r \ge 0$ and $\eps >0$, 
	there is a simple binary matroid $M$ of critical number at least $c$ such that $|M| \ge (\delta - \eps)2^{r(M)}$ and every restriction of $M$ of rank at most $r$ either has critical number at most $1$, or is in $\cN_{\delta'}$ for some $\delta' < \delta$. 
\end{lemma}
\begin{proof}
	We consider the three values of $\delta$ separately. For $\delta = 0$, a matroid $M$ given by Lemma~\ref{bespokegc} with critical number at least $c$ and girth at least $r + 2$ will do, since every rank-$r$ restriction of $M$ is a free matroid and thus has critical number at most $1$. For the other values of $\delta$ we require slightly more technical constructions. 
	
	\textbf{Case 1: $\delta = \tfrac{1}{4}$}. Let $g = r + 2$ and let $s = 2cg$. By Stirling's approximation we have $\binom{2n}{n} \sim \tfrac{1}{\sqrt{\pi n}}2^{2n}$. Let $n \in \bN$ be such that $\binom{2n}{n} \le  \tfrac{2\eps}{gs}  2^{2n}$, and such that there exists a set $X \subseteq \GF(2)^{2n}$, given by Lemma~\ref{bespokegc}, for which $\wt(x) \ge 2n-s$ for each $x \in X$, and $M(X)$ has rank $2n$, girth at least $g$, and critical number at least $c$. Let 	
	\[Y = \lt\{y \in \GF(2)^{2n}: \wt(y) \le n-gs \rt\}.\]
	Let $X',Y' \subseteq \GF(2)^{n+1}$ be defined by $X' = \{ \sqbinom{0}{x}: x \in X \}$ and $Y' = \{\sqbinom{1}{y}: y \in Y\}$. Let $M = M(X' \cup Y')$. First note that $\cn(M) \ge \cn(M(X')) \ge c$. By symmetry of binomial coefficients and the fact that $\binom{2n}{i} \le \binom{2n}{n}$ for each $i$, we have
	\[|M| \ge |Y| \ge \sum_{i = 0}^{n-gs} \binom{2n}{i} \ge \frac{1}{2}\lb 2^{2n} - 2gs\binom{2n}{n} \rb \ge \lb \frac{1}{4} - \eps\rb 2^{2n+1},\]
	so $|M| \ge \lb \tfrac{1}{4}- \eps\rb 2^{r(M)}$. Finally, let $R$ be a restriction of $M$ with $r(R) \le r$. The set $E(R) \cap X'$ contains a $1$-codimensional subspace $S$ of $R$, and since $M(X') = M(X)$ has girth at least $g = r(R) + 2$, the set $S$ is independent in $R$; it follows that $\cn(R) \le 2$. We argue that if $\cn(R) = 2$ then $R \in \cN_0$. 
	
	Let $C$ be an odd circuit of $R$ with $|C - X'| \le 2$, and let $C_X,C_Y \subseteq \GF(2)^{2n}$ be the subsets of $X$ and $Y$ corresponding to $C \cap X'$ and $C \cap Y'$ respectively. Note that $\sum C_X = \sum C_Y$, and $|C_X| + |C_Y| \le r(R) + 1 = g-1$, with $|C_Y| \in \{0,2\}$ and $|C_X|$ odd. By choice of $Y$ we know that $\wt(\sum C_Y) \le 2(n-gs)$. Since every $x \in C_X$ has the form $\bfj + \hat{x}$ where $\bfj$ is the all-ones vector and $\wt(\hat{x}) \le s$, we have $\wt(\sum C_X) \ge 2n - (g-1)s > 2(n-gs) \ge \wt(\sum C_Y)$, a contradiction. Therefore each odd circuit of $R$ contains at least four elements of $E(R)-S$, so $R \in \cN_0$.
		
	\textbf{Case 2: $\delta = \tfrac{1}{2}$}. Let $g = r + 2$ and $n$ be an integer such that there is a set $X \subseteq \GF(2)^n$, given by Lemma~\ref{bespokegc}, so that $M(X)$ has girth at least $g$ and critical number at least $c$. Let $X' = \{\sqbinom{0}{x}: x \in X\}$ and let $Y' = \{\sqbinom{1}{y}: y \in \GF(2)^n\}$. Let $M = M(X' \cup Y')$.
	
	Clearly $\cn(M) \ge \cn(M(X)) \ge c$ and $|M| \ge 2^n \ge \lb \tfrac{1}{2}-\eps \rb 2^{r(M)}$. If $R$ is a restriction of $M$ with $r(R) \le r$, then the set $E(N) \cap X'$ contains a $1$-codimensional subspace $S$ of $R$ and, since $M(X')$ has girth at least $g \ge r(R)+2$, the set $S$ is independent in $R$. It follows that $\cn(R) \le 2$ and $R \notin \cNh$. 
\end{proof}

We can now show that Conjecture~\ref{critc} provides a valid lower bound. 

\begin{theorem}\label{mainlb}
	If $N$ is a simple rank-$r$ binary matroid with critical number $c \ge 2$, then the critical threshold for $N$ is at least $1 - (1-\delta)2^{2-c}$, where $\delta \in \lt\{0,\tfrac{1}{4},\tfrac{1}{2}\rt\}$ is minimal so that $N|S \in \cN_{\delta}$ for some $(c-2)$-codimensional subspace $S$ of $N$.
\end{theorem}
\begin{proof}
	Let $t \in \bZ$ and let $\eps > 0$. By Lemma~\ref{crit2thm} there exists a rank-$n$ matroid $M_0$ for which $\cn(M_0) \ge t$ and $|M_0| \ge (\delta - \eps)2^n$, and such that every restriction $R_0$ of $M_0$ with $r(R_0) \le r$ satisfies either $\cn(R_0) \le 1$ or $R_0 \in \cN_{\delta'}$ for some $\delta' < \delta$. Let $G \cong \PG(n+c-3,2)$ have $M_0$ as a restriction, and let $F_0 = \cl_G(M_0)$. Set $M = G \del (F_0 - E(M_0))$. 
	
	Since $M_0$ is a restriction of $M$, we have $\cn(M) \ge t$. Moreover,
	\begin{align*}
		|M| &= |G| - |F_0| + |M_0| \\
		&\ge (2^{n+c-2}-1) - (2^n-1) + (\delta-\eps)2^{n}\\
		&= (1 - (1 - \delta + \eps)2^{2-c})2^{n+c-2}\\
		&\ge (1 - (1-\delta)2^{2-c} - \eps)2^{r(M)}.
	\end{align*}
	Finally, suppose for a contradiction that $M$ has a restriction $R \cong N$. The set $E(R) \cap F_0$ contains a $(c-2)$-codimensional subspace $S$ of $R$, and $\cn(R|S) \ge \cn(R) - (c-2) = 2$. However, $R|S$ is also a restriction of $M_0$ of rank at most $r$, so either $\cn(R|S) = 1$ or $R|S \in\cN_{\delta'}$ for some $\delta' < \delta$. The former contradicts $\cn(R|S) \ge 2$ and the latter contradicts the minimality of $\delta$. 
\end{proof}

	Finally, we restate and prove Theorem~\ref{richcase}.
	
	\begin{theorem}
	If $N$ is a simple binary matroid of critical number $c \ge 1$ so that $\cn(N \del I) = c$ for every rank-$(r(N)-c+1)$ independent set $I$ of $N$, then the critical threshold for $N$ is $1 - 2^{1-c}$. 
	\end{theorem}
	\begin{proof}
	The upper bound is given by Corollary~\ref{threshupperbound}, which also gives the theorem when $c = 1$. It thus suffices by Theorem~\ref{mainlb} to show that $N$ has no $(c-2)$-codimensional subspace in $\cN_0 \cup \cNq$. Indeed, if $S$ is such a subspace then $N|S$ has an independent $1$-codimensional subspace $I$, so $\cn((N|S) \del I) = 1$. Moreover, $r_N(I) \le r_N(S)-1 = r(N)-c+1$, and $\cn(N \del I) \le 1 + (c-2) < c$, a contradiction. 
	\end{proof}

\section*{Acknowledgements}
We thank the referees for their careful reading of the manuscript and for their useful comments. 

\section*{References}
\newcounter{refs}
\begin{list}{[\arabic{refs}]}
{\usecounter{refs}\setlength{\leftmargin}{10mm}\setlength{\itemsep}{0mm}}

\item\label{abgkm}
P. Allen, J. B\"{o}ttcher, S. Griffiths, Y. Kohayakawa, R. Morris,
The chromatic thresholds of graphs,
Adv. Math, 235 (2013): 261--295.

\item \label{bq}
J.E. Bonin, H. Qin,
Size functions of subgeometry-closed classes of
representable combinatorial geometries,
Discrete Math. 224, (2000) 37--60.

\item\label{bt}
S. Brandt, S. Thomass\'{e}, 
Dense triangle-free graphs are four-colorable, 
to appear, JCTb. 

\item\label{cr}
H.H. Crapo, G.-C. Rota,
On the foundations of combinatorial theory:
Combinatorial geometries, M.I.T. Press, Cambridge, Mass., 1970.

\item\label{gcn}
P. Erd\H os, 
Graph theory and probability, 
Canad. J. Math. 11 (1959) 34--38.

\item\label{erdossimonovits}
P. Erd\H os, M. Simonovits,
On a valence problem in extremal graph theory,
Discrete Math. 5 (1973), 323--334.

\item \label{es}
P. Erd\H os, A.H. Stone,
On the structure of linear graphs,
Bull. Amer. Math. Soc. 52, (1946) 1087-1091.

\item \label{fk2}
H. Furstenberg, Y. Katznelson,
An ergodic Szemeredi theorem for IP-systems and combinatorial theory,
Journal d'Analyse Math\' ematique 45, (1985) 117--168.

\item\label{gl}
W. Goddard, J. Lyle, 
Dense graphs with small clique number, 
J. Graph Theory 66 (2011) no. 4, 319--331.

\item\label{green04}
B. Green, 
Finite field models in additive combinatorics,
Surveys in combinatorics 2005, London Mathematical Society Lecture Note Series 327 (Cambridge University Press, Cambridge, 2005), pp. 1--27. 

\item\label{green}
B. Green,
A Szemer\'{e}di-type regularity lemma in abelian groups, with applications,
Geometric \& Functional Analysis GAFA 15 (2005), 340--376.

\item\label{gn12}
J. Geelen, P. Nelson, 
An analogue of the Erd\H os-Stone theorem for finite geometries, 
Combinatorica, in press. 

\item\label{gn14}
J. Geelen, P. Nelson, 
Odd circuits in dense binary matroids, 
Combinatorica, to appear. 

\item\label{29}
R. H\"aggkvist,
Odd cycles of specified length in nonbipartite graphs, 
Graph theory (Cambridge, 1981), 89--99.

\item\label{j81}
F. Jaeger, 
A constructive approach to the critical problem for matroids, 
Eur. J. Combin. 2 (1981), 137--144.

\item\label{tl}
T. \L uczak, S. Thomass\' e, 
Coloring dense graphs via VC-dimension,
arXiv:1007.1670 [math.CO]

\item \label{oxley}
J. G. Oxley, 
Matroid Theory,
Oxford University Press, New York (2011).

\item\label{tv06}
T. C. Tao, V. H. Vu, 
Additive Combinatorics, 
Cambridge Studies in Advanced Mathematics, 105, 
Cambridge University Press, Cambridge (2006). 

\item\label{t02}
C. Thomassen, 
On the chromatic number of triangle-free graphs of large minimum degree,
Combinatorica 22 (2002), 591--596


\item\label{welsh}
D.J.A. Welsh,
Matroid Theory, Academic Press, London (1976). Reprinted 2010,
Dover, Mineola.

\end{list}		
\end{document}